\documentclass{amsart}
\usepackage{setspace}
\usepackage{amsmath,amsfonts,amscd,amssymb}
\usepackage{graphicx}
\usepackage[all]{xy}
\usepackage[colorlinks,
            linkcolor=red,
            anchorcolor=blue,
            citecolor=green
            ]{hyperref}

\newcommand{\mC}{\mathcal{C}}
\newcommand{\mK}{\mathcal{K}}

\newcommand{\mG}{\mathcal{G}}

\newcommand{\ZZ}{\mathbb{Z}}

\newcommand{\CC}{\mathbb{C}}
\newcommand{\QQ}{\mathbb{Q}}
\newcommand{\RR}{\mathbb{R}}

\newcommand{\Map}{\mbox{Map}}

\newcommand{\rot}{\mbox{rot}}

\newcommand{\ev}{\mbox{ev}}

\newcommand{\w}{\omega}

\newtheorem{thm}{Theorem}[section]
\newtheorem{dfn}[thm]{Definition}
\newtheorem{cor}[thm]{Corollary}
\newtheorem{lma}[thm]{Lemma}
\newtheorem{prp}[thm]{Proposition}

\newtheorem{clm}[thm]{Claim}
\newtheorem{rmk}[thm]{Remark}

\begin{document}

\title
{The symplectic mapping class group of $\CC P^2  \# n{\overline {\CC P^2}}$ with  $n\leq4$}

\author{ Jun Li, Tian-Jun Li, Weiwei Wu}

\address{School  of Mathematics\\  University of Minnesota\\ Minneapolis, MN 55455}
\email{lixx1727@umn.edu}

\address{School  of Mathematics\\  University of Minnesota\\ Minneapolis, MN 55455}
\email{tjli@math.umn.edu}

\address{Department of Mathematics\\  Michigan State University\\ East Lansing, MI}
\email{weiwei@math.msu.edu }

\date{\today}

\begin{abstract}
In this paper we  prove that the Torelli part of the  symplectomorphism groups
of
the $n$-point ($n\leq 4$) blow-ups of the projective plane is trivial.
Consequently, we determine the symplectic mapping class group. It is generated by reflections  on $K_{\omega}-
$spherical class with zero $\omega$ area.
\end{abstract}

\maketitle
\tableofcontents

\section{Introduction}

A symplectic manifold $(X,\omega)$ is an even dimensional manifold $X$ with a closed, nondegenerate two form $\omega$.
The symplectomorphism group of $(X,\omega)$, denoted by $Symp(X,\omega)$, is the group of diffeomorphisms $\phi$ of $M$
which preserve $\omega$, and is given the $C^{\infty}$-topology. $Symp(X,\omega)$ is an
infinite dimensional Fr\'echet Lie group.

For a closed $4-$dimensional symplectic manifold $(X,\omega)$, since Gromov's
 work \cite{Gro85}, the homotopy type of $Symp(X,\omega)$ has attracted much interest over the past 30 years.
For the special case of some monotone $4-$manifolds, the (rational) homotopy of
$Symp(X,\omega)$ was fully computed in \cite{Gro85, AM99, Eva11}. However, for an arbitrary symplectic 4 manifold,
 the complication grows drastically: for $S^2\times S^2$,
see \cite{Abr98, AM99, Anj02}; and \cite{AP12} for other instances.

The goal of this note is modest: for some rational $4-$manifolds, we compute $\pi_0(Symp(X,\omega))$,
which is the symplectic mapping class group (denoted as SMC for short). In the cases we consider, the homological action of
$Symp(X,\omega)$ is already known in \cite{LW11}. Therefore it suffices  to describe  $\pi_0(Symp_h(X,\omega))$,
which is the subgroup of $Symp(X,\omega)$ acting trivially on homology, namely, its Torelli part.

\begin{thm}\label{main}

 $Symp_h(X,\omega)$ is connected for  $X= \CC P^2  \# 4 {\overline {\CC P^2}}$ with arbitrary symplectic form $\omega$.
\end{thm}

The cases $S^2\times S^2$ and  $(\CC P^2  \# k {\overline {\CC P^2}})$ with  $k\leq 3$ are known before.
Our approach actually works in a uniform way for all $k\leq 4$ (See discussions
in remark \ref{kleq3}). One also note that Theorem \ref{main} is not true in general for
$k\geq5$, see Seidel's famous example in \cite{Sei08}.

Our strategy   is based on Evans' beautiful approach
in \cite{Eva11} by systematically exploring the geometry of certain stable configuration of symplectic spheres
(a related approach first appeared in Abreu's paper \cite{Abr98}).
It is summarized by the following diagram:

\begin{equation} \label{summary}
\begin{CD}
Symp_c(U) @>>>  Stab^1(C) @>>> Stab^0(C) @>>> Stab(C) @>>> Symp_h(X) \\
@.@. @VVV @VVV @VVV \\
@.@. \mG(C) @. Symp(C) @. \mC_0
\end{CD}
\end{equation}
Here $\mC_0$ is the space of a full stable standard configuration of fixed homological type.
Every other term in diagram \eqref{summary} is a group associated to $C \in \mC_0 $, and $U=X\setminus C$.
Now we give the definition of stable standard spherical configurations and the groups will be discussed
later in section \ref{group}.

\begin{dfn}\label{stable conf}
Given a symplectic 4-manifold $(X,\omega)$, we call an ordered finite collection of symplectic
spheres $\{C_i, i=1, ...,n\}$  a spherical symplectic configuration, or simply a {\bf configuration}  if

1. for any pair  $i,j$ with $i\ne j$, $[C_i]\ne [C_j]$ and $[C_i] \cdot [C_j]=0$ or $1$.

2. they are simultaneously $J-$holomorphic for some $J \in \mathcal J_{\omega}$.

3. $C=\bigcup C_i$ is connected.

We will often use $C$ to denote  the configuration. The homological type of $C$ refers to the set of homology classes $\{[C_i]\}$.

Further, a   configuration  is called
\begin{itemize}

\item   {\bf standard} if the components intersect
$\omega$-orthogonally
 at every intersection point of the configuration.
 Denote by $\mC_0$  the space of  standard configurations having the same homology type as $C$.

\item {\bf stable} if  $[C_i] \cdot [C_i] \geq -1$ for each $i$.

\item  {\bf full} if  $H^2(X,C;\RR)=0$.

\end{itemize}
\end{dfn}

It is shown in \cite{LW11} that  for a rational manifold, the homological action of $Symp(X,\omega)$ is
generated by Lagrangian Dehn twists. Therefore, Theorem \ref{main} implies:
\begin{cor}
For a rational manifold with Euler number up to $7$, the SMC is a finite group generated by
Lagrangian Dehn twists. Moreover, a generating set corresponds to
a finite set of $K_{\omega}-$null spherical classes with   zero $\omega-$area.
In particular, SMC is  trivial for generic choice of $\omega$.
\end{cor}

It is shown in \cite{BLW12} that the following proposition holds:

\begin{prp}\label{c:TTS2} Suppose  $(X^{4}, \omega)$ is a symplectic rational manifold. Then $Symp_h(X, \omega)$ acts
transitively on the space of

\begin{itemize}
  \item homologous Lagrangian spheres
  \item homologous symplectic $-2$-spheres
 \item   $\ZZ_2$-homologous Lagrangian $\RR P^2$'s and homologous symplectic $-4$-spheres if $b^{-}_{2}(X) \leq 8$

 \end{itemize}
\end{prp}

Hence we  also have the following corollary:
\begin{cor}

For a rational manifold with Euler number up to $7$,
the space  of \\
$\bullet$ homologous Lagrangian spheres, \\
$\bullet$ $\ZZ_2$- homologous Lagrangian  $\RR P^2$,\\
$\bullet$ homologous $-2$  symplectic spheres,\\
$\bullet$ homologous $-4$  symplectic spheres,\\
is connected.
\end{cor}






{\bf Acknowledgements:} Our indebtedness
 to Jonathan Evan's  illuminating paper \cite{Eva11} is throughout and evident. We would
like to thank Martin Pinsonnault for sharing his insights on an upcoming project towards all
rational homotopy groups
 of  $Symp(\CC P^2  \# 4{\overline {\CC P^2}}, \omega)$, and for his comments.   We are also grateful to Robert Gompf for useful discussions.
We thank an anonymous referee for the careful reading and many useful comments which greatly improved our exposition.
 T.-J. Li and W. Wu are supported
 by NSF Focused Research Grants DMS-0244663, W.Wu is supported by AMS-Simons travel funds.

\section{Analyzing the diagram}

 We analyze the   diagram \eqref{summary} and  derive a
criterion for the connectedness of $Symp_h(X, \omega)$ in Corollary \ref{criterion}.

\subsection{Groups associated to a configuration}\label{group}
Let $C$ be a  configuration in $X$.
We first introduce the groups appearing in \eqref{summary}:\\

{\bf Subgroups of $Symp_h(X, \omega)$}

Recall that $Symp_h(X, \omega)$ is the group of symplectomorphisms of $(X, \omega)$ which acts trivially on $H_*(X, \ZZ)$.

$\bullet$ $Stab(C) \subset Symp_h(X, \omega)$ is the subgroup of symplectomorphisms fixing  $C$ setwise, but not necessarily pointwise.

$\bullet$ $Stab^0(C)\subset Stab(C)$ is the subgroup the group  fixing  $C$ pointwise.

$\bullet$  $Stab^1(C) \subset Stab^0(C)$ is subgroup fixing  $C$ pointwise and acting trivially on the normal bundles of its
components.\\

{\bf $Symp_c(U)$ for the complement $U$}

$Symp_c(U)$ is the group of compactly supported symplectomorphisms  of $(U, \omega|_U)$,
where $U=X\setminus C$ and the form $\omega|_U$ is the inherited form on $U$ from $X$. It is topologised in this way:
let $(U,\omega)$ be a non-compact symplectic manifold and let $\mK$ be the set
of compact subsets of $U$. For each $K\in\mK$ let $Symp_K(W)$ denote the group
of symplectomorphisms of $U$ supported in $K$, with the topology of
$\mC^{\infty}$-convergence. The group $Symp_c(U,\omega)$ of
compactly-supported symplectomorphisms of $(U,\omega)$ is topologised as the
direct limit of $Symp_K(W)$ under inclusions.\\

{\bf $Symp(C)$  and $\mG(C)$ for the configuration $C$}

Given  a configuration of embedded
symplectic spheres $C=C_1\cup\cdots\cup C_n\subset X$ in a 4-manifold, let
$I$ denote the set of intersection points amongst the components. Suppose that
there is no triple intersection amongst components and that all intersections
are transverse.  Let   $k_i$ denote the cardinality of  $I\cap C_i$, which is the number of intersection of points on $C_i$.

    The group $Symp(C)$ of symplectomorphisms of $C$ fixing the
components of $C$ is  the product
$\prod_{i=1}^nSymp(C_i,I\cap C_i)$.
Here $Symp(C_i,I\cap C_i)$ denotes the group of symplectomorphisms of  $C_i$
fixing   the  intersection points $I\cap C_i$.  Since  each $C_i$ is a $2-$sphere and  $Symp(S^2)$ acts transitivity on $N-$tuples of distinct points in $S^2$, we can write $Symp(C_i, I\cap C_i)$ as
$Symp(S^2, k_i)$. Thus
\begin{equation}\label{sympc}
 Symp(C)\cong \prod_{i=1}^nSymp(S^2,k_i)
\end{equation}
As shown in \cite{Eva11} we have:
\begin{equation} \label{sympk}
Symp(S^2,1)\simeq S^1;\hspace{5mm}
Symp(S^2,2)\simeq  S^1;\hspace{5mm}
Symp(S^2,3)\simeq {\star};
\end{equation}\\
where $\simeq$ means homotopy equivalence.
And when $k=1,2$, the $S^1$ on the right  can be taken to be the  loop of  a Hamiltonian circle action fixing the $k$ points.

The symplectic gauge group  $\mG(C)$ is the product $\prod_{i=1}^n\mG_{k_i}(C_i)$. Here $\mG_{k_i}(C_i)$ denotes
the group of symplectic gauge transformations of the symplectic normal bundle
to $C_i\subset X$ which are equal to the identity at the $k_i$ intersection points.
Also shown in \cite{Eva11}:
\begin{equation} \label{gau}
\mG_0(S^2) \simeq S^1 ; \hspace{5mm}  \mG_1(S^2) \simeq \star; \hspace{5mm}  \mG_k(S^2) \simeq \ZZ^{k-1},\ k>1.
\end{equation}
Since we assume the configuration is connected, each $k_i\geq 1$. Thus by \eqref{gau}, we have
\begin{equation}\label{mG}
 \pi_0(\mG(C))= \oplus_{i=1}^n \pi_0(\mG_{k_i}(S^2))=\oplus_{i=1}^n  \mathbb Z^{k_i-1}
\end{equation}
It is useful to describe  a canonical  set of $k_i$ generators for  $\mG_{k_i}(C_i)$.
For each intersection point $y\in I\cap C_i$, the evaluation map
$$ev_y:  \mG_{k_i}(C_i)\to SL(2,  \RR)$$ is a homotopy fibration, and hence it induces a map
$\ZZ=\pi_1(SL(2, \RR))\to \pi_0(\mG_{k_i}(C_i))$.
Let $g_{C_i}(y)\in \pi_0(\mG_{k_i}(C_i))$ denote the image of $1\in \ZZ$.



\subsection{Reduction  to the connectedness of $Stab(C)$}

The aim of this subsection is to show

\begin{prp}\label{stab connected}
$Symp_h(X, \omega)$ is connected if there is a full, stable,  standard configuration $C$
with connected
$Stab(C)$.
\end{prp}

This is derived from the right end of   diagram \eqref{summary} for a full, stable, standard configuration $C$:
\begin{equation}\label{stab}
Stab(C) \rightarrow Symp_h(X, \omega)\rightarrow \mC_0
\end{equation}

Recall that $\mC_0$ is the space of
standard  configurations having the homology type of $C$. We will show \eqref{summary} is a homotopy fibration and
 $\mC_0$ is connected.

We first review certain general facts regarding these configurations which are well-known
to experts. By \cite{LW11}, we have the following fact.

\begin{lma} \label{symplectic connectedness} Let $(M,\omega)$ be a
symplectic 4-manifold and $C$ a stable  configuration $\cup_i  C_i$.  Then
 there is a path connected Baire subset $\mathcal T_D$ of $\mathcal J_\omega\times M_{d(C_i)}$ such that
 a pair $(J, \Omega)$ lies in $\mathcal T_D$
 if and only if  there is a unique  embedded $J-$holomorphic configuration having the same homological type as $C$
 with the $i-$th component containing
 $\Omega_i$.
 \end{lma}

\begin{lma}\label{c0conn}
Assume $C$ is a stable, standard configuration.  The space $\mC_0$ of  standard configurations  having the homology type of $C$ is path connected.
\end{lma}





\begin{proof}
Consider $\mC$, the space of configurations as in Definition \ref{stable conf}.
By  Lemma \ref{symplectic connectedness}  we see that the space $\mC$ is connected. Using a Gompf isotopy argument, it is shown in \cite{Eva11} that  the inclusion  $\iota:\mC_0\to \mC$ is a weak homotopy equivalence.
Therefore, $\mC_0$ is also connected.

\end{proof}

With $C$ being full, the following lemma holds:

\begin{lma} \label{transitive}  If the   stable,  standard configuration $C$ is also full, then
$Symp_h(X, \omega)$ acts transitively on $\mC_0$. In particular, \eqref{stab} is a
homotopy fibration.
\end{lma}

\begin{proof}
 From Lemma \ref{c0conn} any $C_1,C_2\in\mC_0$ are isotopic through standard
configurations. The property that the configurations are {\bf symplectically
orthogonal} where they intersect, together with  the {\bf vanishing} of  $H^2(X,C;\RR)$, allows us to
extend such an isotopy to a global homologically trivial symplectomorphism of $X$ (by Banyaga's
symplectic isotopy extension theorem, see \cite{MS05}, Theorem 3.19).
So we have shown that the action of $Symp_h(X, \omega)$  on the connected space $\mC_0$ is transitive by  establishing  the $1-$dimensional
homotopy lifting property of the map $Symp_h(X, \omega)\to \mC_0$.
By a finite dimensional version of this argument
(or   Theorem A in \cite{Pai60}), we conclude that
\eqref{stab} is a
homotopy fibration.
\end{proof}

{\bf Proof of Proposition \ref{stab connected}}

Since \eqref{stab} is a homotopy fibration by Lemma \ref{transitive}, we have the associated   homotopy long exact sequence.
Because of  the connectedness of $\mC_0$ as shown in Lemma  \ref{c0conn}, the connectedness of
$Stab(C)$ implies the connectedness of $Symp_h(X, \omega)$.
Therefore, we have \ref{stab connected} as the reduction of our problem.

\subsection{Reduction to the surjectivity of $\psi$: $\pi_1(Symp(C))\rightarrow\pi_0(Stab^0(C))$}

To investigate the connectedness of  $Stab(C)$, considering the action of  $Stab(C)$ on $C$ and  the following portion of diagram \ref{summary}
 which appeared in \cite{Eva11} and \cite{AP12}:

\begin{equation}\label{stab0}
 Stab^{0}(C) \rightarrow Stab(C)\rightarrow Symp(C)
\end{equation}
 The following lemma  already appeared in \cite{Eva11} and was explained to the authors by
 J. D. Evans\footnote{Private communications.}. We here include more details for readers' convenience.

\begin{lma}\label{tree}

This diagram  \eqref{stab0}  is a homotopy fibration when
$C$ is a simply-connected standard configuration.

\end{lma}

\begin{proof}
 First we show   $Stab(C)\rightarrow Symp(C)$ is surjective.

Recall that at each intersection point between two different components $\{x_{ij} \}= C_i\cap C_j$, the two components are
symplectically orthogonal to each other in a Darboux chart containing $x_{ij}$.
For convenience of exposition define the \textit{level} of components as follows: let $C_1$ be the unique
component of level $1$, and the level-$k$ components are defined as those
intersects components in level $k-1$ but does not belong to any lower levels.
This is well-defined again because of the simply-connectedness assumption.

An element in $Symp(C)$ is the composition of  Hamiltonian
diffeomorphism $\phi_i$ on each component $C_i$, because of the
simply connectedness of sphere. We start with endowing $C_1$ with a
Hamiltonian function $f_1$ generating $\phi_1$.  Let $C_i^2$ be
curves on level $2$.  Because $C^2_i$ intersects $C_1$
$\omega$-orthogonally, we can find a symplectic neighborhood $U_1$
of $C_1$, identified as a neighborhood of zero section of the normal
bundle, so that $U_1\cap C_i$ consists of finitely many fibers.
Pull-back $f_1$ by the projection $\pi$ of the normal bundle and
multiply a cut-off function $ \rho(r), \rho(r)=1, r\leq \epsilon \ll
1; \rho(r)=0, r\geq 2\epsilon$. Here $r$ is the radius in the fiber
direction. Denote by $\bar {\phi}_1$ the symplectomorphism generated
by this cut-off.
  Notice that $\bar {\phi}_1$ creates an extra Hamiltonian diffeomorphism $\epsilon_j$ on each component $C_j$ of level $2$,
 and we denote $\phi_j'=\phi_j\circ \epsilon_j^{-1}$ for $C_j$ belonging to level $2$.

One proceeds by induction on the level $k$.  Notice one could always
choose a Hamiltonian function $f_i$ on a component $C_i$ on level
$k$ which generates $\phi_i'$ with the property that
$f_i(x_{il})=0$. Here $C_l$ is the component of level $k-1$
intersecting $C_i$.  We emphasize this can be done because the
component $C_l$ on level $k-1$ which intersects $C_i$ is unique (and
that the intersection is a single point) due to the simply
connectedness assumption, and we do not restrict the value on any
other intersections of $C_i$ and components of level $k+1$.
Therefore we only fix the value of $f_i$ at a single point.

 One then again use the pull-back on the symplectic neighborhood and
 cut-off along the fiber direction to get a Hamiltonian function $H_i$
 which generates a diffeomorphism $\bar {\phi_i}$ supported on
the neighborhood of $C_i$. We note that $d(\pi^*f_1\cdot
\rho(r))|_{F_x}=0$ whenever $f_1(x)=0$, where $F_x$ is the normal
fiber over the point $x \in C_1$.   Hence $dH_i |_{C_l}=0$ since
$f_i(x_{il})=0$ as prescribed earlier, which means action of $\bar
{\phi_i}$ on $C_l$ is trivial. Taking the composition $\phi$ of all
these $\bar {\phi_i}'s$, $\phi$ is supported on a neighborhood of
$C$ and equals $\phi_i$ when restricted to $C_i$.

The transitivity of the action of $Stab(C)$ on $Symp(C)$ follows easily.
For any two maps $\phi_1, \phi_2 \in Symp(C),$ $ \phi_2\phi_1^{-1}  \in Symp(C)$. We can extend $\phi_2\phi_1^{-1}$ to $Stab(C)$.
Then this extended $\phi_2\phi_1^{-1}$ maps $\phi_1$ to $\phi_2$.

Now symplectic isotopy theorem (or Theorem A in \cite{Pai60}) for
the surjective map $Stab(C)\rightarrow Symp(C)$ proves the diagram
\eqref{stab0} is a fibration.

\end{proof}

Now we can establish the connectedness of $Stab(C)$ under the following assumptions:

\begin{prp}\label{surj}
 Let $(X,\omega)$ be a symplectic 4-manifold, and $C$  a simply-connected, full,  stable,  standard configuration.
 If  each component of $C$ has no more than 3 intersection points,
 then the surjectivity of the connecting map $\psi$: $\pi_1(Symp(C))\rightarrow\pi_0(Stab^0(C))$
 implies the connectedness of $Stab(C)$.
\end{prp}

\begin{proof}
Since we assume that  each component of $C$ has no more than 3 intersection points, it follows from  \eqref{sympk}  and \eqref{sympc} that
$\pi_0(Symp(C))=1$.

By Lemma \ref{tree} we have   the homotopy long exact sequence associated to \eqref{stab0},
\[\cdots\rightarrow\pi_1(Symp(C))\overset{\psi}\rightarrow\pi_0(Stab^0(C))\rightarrow\pi_0(Stab(C))\rightarrow\pi_0(Symp(C))\]
Then the surjectivity of $\psi$ implies that
$Stab(C)$ is connected.

\end{proof}

\subsection{Three types of configurations}

Next we investigate when  the  map $\psi$: $\pi_1(Symp(C))\rightarrow\pi_0(Stab^0(C))$ is surjective. For this purpose we observe
that an element of $Stab^{0}(C)$ induces an automorphisms on the normal bundle of $C$. Thus  we further have the following
homotopy fibration  appeared in \cite{Eva11} and \cite{AP12}:

\begin{equation}\label{stab1}
Stab^{1}(C)\rightarrow Stab^{0}(C) \rightarrow \mG(C)
\end{equation}
In particular, there is the associated map   $\iota: \pi_0(Stab^0(C))\rightarrow \pi_0(\mG)(C)$.
Consider
 the composition map \[\bar{\psi}=\iota\circ \psi:\pi_1(Symp(C))\rightarrow\pi_0(Stab^0(C))\rightarrow \pi_0(\mG(C)).\]
Notice that $\pi_0(\mG(C))$ inherits a  group structure from     $\mG(C)$ and $\bar \psi$ is a group homomorphism.
As shown in \cite{Eva11}, $\bar \psi$ can be computed explicitly.

When $k_i\geq 3$,  $\pi_1(Symp(S^2, k))$ is trivial  by \eqref{sympk}.
When $k_i=1,2$,  $Symp(C_i,  I\cap C_i)$  is homotopic to
 the loop
of  a Hamiltonian circle action on $C_i$  fixing  the $k_i$ points. Denote such a loop
 by $(\phi_i)_{t}$. Observe that $(\phi_i)_t$ is
a generator of  $\pi_1(Symp(C_i, I \cap C_i))=\mathbb Z$.
Recall that for each component $C_j$ there is a canonical set of generators $\{g_{C_j}(y), y\in I\cap C_j\}$ for $\mG_{k_j}(C_j)$, introduced at the end of 2.1.
The following is Lemma 4.1 in \cite{Eva11}

\begin{lma}\label{localcomp} Suppose $C_i$ is a component with $k_i=1, 2$.
The image of $(\phi_i)_{2\pi}\in Symp(C_i, I\cap C_i)$ under $\bar{\psi}$ is described as follows.
\begin{itemize}
\item   if $k_i=1$ and  $C_j$ is the only component intersecting $C_i$ with $\{x\}=C_i\cap C_j$,  then  $(\phi_i)_{2\pi}$ is sent to     $$g_{C_j}(x)$$  in
the factor subgroup  $ \pi_0(\mG_{k_j}(C_j))$ of  $\pi_0(\mG(C))$.
\item  if $k_i=2$ and
$x\in C_i\cap C_j$, $y\in C_i\cap C_l$, then   $(\phi_i)_{2\pi}$ is sent to
$$(g_{C_j}(x),g_{C_k}(y))$$ in the factor subgroup $\pi_0(\mG_{k_j}(C_j))\times\pi_0(\mG_{k_l}(C_l))$ of $\pi_0(\mG(C))$.
 \end{itemize}
\end{lma}

 Use Lemma \ref{localcomp} we will show that $\bar \psi$ is surjective for the following configurations.

\begin{dfn} \label{three graphs} Introduce three types of  configurations (see  Figure 1 for examples).
\begin{itemize}
\item (type I)  $C=\bigcup_{1}^{n} C_i$ is called a chain, or a type I configuration,  if  $k_1=k_n=1$  and $k_j=2,$  $ 2 \leq j \leq n-1$.

\item  (type II) Suppose $C=\bigcup_{1}^{n} C_i$ is a chain. $C'=C\cup \overline {C_p}$ is called a type II configuration if
the sphere  $\overline{C_p}$ is attached  to $C_p$ at exactly one point for some  $p$ with $2 \leq p \leq n-1$.

\item (type III) Suppose  $C'=C\cup \overline {C_p}$ is  a type II  configuration.  $C''=C'\cup \overline {C_q}$ is called a type III configuration  if
the
sphere $\overline{C_q}$ is attached to $C_q$ at exactly one point for some $q$ with $ 2 \leq q \leq n-1$ and $ q\neq p$.
\end{itemize}

\end{dfn}

\begin{figure}[h!]
\[
\xy
(0,0)*{};(10,0)*{} **\dir{-};
(5,3)*{C_{1}};
(17,-14)*{};(10,0)*{} **\dir{-};
(18,-10)*{C_2};
(17,4)*{};(10,-10)*{} **\dir{-};
(19,0)*{\overline{C_2}};
(-7,-14)*{};(0,-28)*{} **\dir{-};
(-5,-25)*{C_{5}};
(10,-28)*{};(17,-14)*{} **\dir{-};
(15,-25)*{C_{3}};
(10,-28)*{};(0,-28)*{} **\dir{-};
(8,-25)*{C_{4}};
(5,-13)*{II};
(-50,0)*{};(-40,0)*{} **\dir{-};
(-45,3)*{C_{1}};
(-33,-14)*{};(-40,0)*{} **\dir{-};
(-38,-10)*{C_2};
(-57,-14)*{};(-50,-28)*{} **\dir{-};
(-55,-25)*{C_{5}};
(-40,-28)*{};(-33,-14)*{} **\dir{-};
(-35,-25)*{C_{3}};
(-40,-28)*{};(-50,-28)*{} **\dir{-};
(-42,-24)*{C_{4}};
(-45,-13)*{I};
(50,0)*{};(60,0)*{} **\dir{-};
(55,3)*{C_{1}};
(67,-14)*{};(60,0)*{} **\dir{-};
(68,-10)*{C_2};
(67,4)*{};(60,-10)*{} **\dir{-};
(69,0)*{\overline{C_2}};
(43,-14)*{};(50,-28)*{} **\dir{-};
(45,-25)*{C_{5}};
(60,-28)*{};(67,-14)*{} **\dir{-};
(65,-25)*{C_{3}};
(60,-28)*{};(50,-28)*{} **\dir{-};
(58,-24)*{C_{4}};
(55,-17)*{};(55,-30)*{} **\dir{-};
(52,-20)*{\overline{C_4}};
(55,-13)*{III};

\endxy
\]
\caption{}
\end{figure}

\begin{lma}\label{chain}

  $\bar\psi$ is surjective for a type I  or II configuration and an  isomorphism for a type  III configuration.

\end{lma}

\begin{proof}

We first prove the surjectivity for a type I configuration  $C=\bigcup_{1}^{n} C_i$.
In this case, there are $n-1$ intersection points $x_1, ..., x_{n-1}$  in total with
$$I\cap C_1=\{x_1\}, \quad  I\cap C_n=\{x_{n-1}\}, \quad  I\cap C_i=\{x_{i-1}, x_i\},\,\,  i=2, ..., n.$$
Notice that  $\pi_1(Symp(C_i, k_i))=\ZZ$ for each $i=1, ..., n$.
Notice also that $\pi_0(\mG_{k_i}(C_i))=\ZZ$
for each $i$ for $i=2, ..., n-1$, and  $\pi_0(\mG_{k_1}(C_1))$ and $\pi_0(\mG_{k_n}(C_n))$ are trivial.
Thus the homomorphism $\bar \psi_C$ associated to $C$ is of the form $\ZZ^n\to  \ZZ^{n-2}$.

For each $i=1, ..., n$, denote   the generator $(\phi_i)_t$ of $\pi_1(Symp(C_i, k_i))=\ZZ$ by $\rot(i)$.
For each $i=2, ..., n-1$, denote  by   $g_i(i-1)$ and $g_i(i)$ the generators  $g_{C_i}(x_{i-1})$ and
$g_{C_i}(x_i)$ of   $\pi_0(\mG_{2}(C_i))=\ZZ$.

Then by Lemma  \ref{localcomp} the homomorphism  $\bar \psi_C$   is described by

\begin{equation}\label{psi}
\begin{aligned}
 &&& \rot(1)&&\rightarrow && g_2(1),\\
&&&\rot(2)&&\rightarrow && (0,g_3(2)),\\
&\bar \psi_C: &&\rot(j)&& \rightarrow && (g_{j-1}(j-1),g_{j+1}(j)), && 3\leq j \leq n-2\\
&&&\rot(n-1) &&\rightarrow && (g_{n-2}(n-2),0)\\
&&&\rot(n)&& \rightarrow && g_{n-1}(n-1)\\
 \end{aligned}
\end{equation}
Choose the bases of  $\pi_1(Symp(C_i))$ and $\pi_0(\mG(C ))$    to be
$$\{ \rot(1),\cdots, \rot(n)\}$$
and  $$\{ g_2(2), g_3(3), g_4(4),\cdots, g_{n-1}(n-1)\},$$  respectively.
Notice that
$ g_i(i-1)=\pm g_i(i)$,      then by \eqref{psi},
$\bar \psi_C $  is represented by the following $ (n-2)\times n$ matrix if  we drop the possible negative sign
for each entry,

$$\begin{bmatrix}
      1&0&1\\
      0& 1&0&1   \\
      0&0&1&0&1 & 0 \\
      &&& \ddots &\ddots& \ddots \\
      & & &&1 & 0 & 1&0&0\\
      & & & && 1 &0&1&0\\
      & & & & && 1&0&1

\end{bmatrix} $$
Observe that  the first $n-2$ minor as
a $ (n-2)\times (n-2)$ is upper triangular matrix
whose determinant is $\pm 1$.
This shows that $\bar \psi_C$ is surjective.

For a type II configuration $C'=C\cup \overline {C_p}$,  let $\bar x_p$ be the intersection of $C_p$ and $\overline{ C_p}$.
Notice that $\pi_1(Symp(C'))=\ZZ^n$ as in the case of $C$, with the $\ZZ$ summand from $C_p$ replaced by  a $\ZZ$ summand  from $\overline{C_p}$.
Notice also that $\pi_0(\mG(C'))=\ZZ^{n-1}$ with the extra $\ZZ$ summand coming from the new intersection point $\bar x_p$  in $C_p$.
Denote by
$\rot(\bar p)$   the generator of  $\pi_1(Symp(\overline{C_p}, \bar x_p))$. Denote by $g'_p(p)$   the generator $g_{C_p}(\bar x_p)$
of $\pi_0(\mG_{3}(C_p))$.  By Lemma \ref{localcomp},
the homomorphism   $\bar \psi_{C'} $ is of the form $\ZZ^n\to \ZZ^{n-1}$, and it differs from $\bar \psi_C$
 as in \eqref{psi} :
  \begin{equation}
   \begin{aligned}
 &\rot(p)=0\\
&\rot(\bar p) \rightarrow g'_p(p)
  \end{aligned}
 \end{equation}
It is not hard to see that  $\bar \psi_{C'}$ is again surjective.
We illustrate by the type II configuration in Figure 1.
With respect to the bases $$\{ \rot(1),\rot(\bar 2),  \rot(3),  \rot(4),  \rot(5)\}\quad
\hbox{and} \quad   \{ g_2(2), g'_2(2), g_3(3), g_4(4)\},$$
$\bar \psi_{C'}$ is represented by the following  $4\times 5$ matrix (if  we drop the possible  negative sign),

$$\begin{bmatrix}
      1&0&1\\
      0&1&0&0   \\
      0&0&1&0&0 \\

      & & 0&1&1\\

\end{bmatrix} $$

For a type III configuration $C''=C'\cup \overline {C_q}=C\cup \overline {C_p}\cup  \overline {C_q}$,
 observe first that $\pi_1(Symp(C''))=\ZZ^n$ and  $\pi_0(\mG(C')=\ZZ^{n}$.
  By Lemma \ref{localcomp},  we can describe $\bar \psi_{C''}:\ZZ^n\to \ZZ^n$
 similar to the case of the type II configuration $C'$.
Precisely, $\bar \psi_{C''}$ differs from $\bar \psi_C$ in
 \eqref{psi} as follows:

\begin{equation}
\begin{aligned}
&\rot(p)= \rot(q)=0\\
&\rot(\bar p)\rightarrow  g'_p(p)\\
&\rot(\bar q)\rightarrow  g'_q(q)\\
\end{aligned}
\end{equation}
It is easy to see that  $\bar \psi_{C''}$ is  an isomorphism in this case.
We illustrate by the type III configuration in Figure 1. With respect to the bases
$$\{ \rot(1),\rot(\bar 2), \rot(3), \rot(\bar 4), \rot(5)\}\quad
\hbox{and} \quad   \{ g_2(2), g'_2(2), g_3(3), g'_4(4), g_4(4)\},$$
$\bar \psi_{C''}$ is represented by the following  square  matrix (if  we drop the possible negative sign),
$$\begin{bmatrix}
      1&0&1\\
      0&1&0&0   \\
      0&0&1&0&0 \\
      0&0&0&1&0 \\
      & & 0&0&1\\

\end{bmatrix} $$


\end{proof}


\subsection{Criterion}

Finally, we arrive at the following criterion for the connectedness of $Symp_h(X, \omega)$.

\begin{cor} \label{criterion}
 Suppose  a stable, standard configuration $C$ is type I, II or III, and it is full.
If  $Symp_c(U)$ is connected, then $Symp_h(X, \omega)$ is  connnected.
\end{cor}

\begin{proof}

By Lemma 5.2 in \cite{Eva11},
$Symp_{c}(U)$ is weakly
homotopy equivalent to $Stab^{1}(C)$. So by our assumption that $Symp_c(U)$ being connected,  $Stab^{1}(C)$ is also connected.
Therefore the map $\iota:  \pi_0(Stab^0(C))\rightarrow \pi_0(\mG)(C)$
associated to
the homotopy fibration \eqref{stab1} is a group isomorphism.
Now we have $\psi_C=\bar \psi_C$.

Since $C$ is type I, II or III, by Lemma \ref{chain}, $\psi_C$ is surjective.
Notice that any  type I,  II, or III configuration is simply-connected.
By the assumption of $C$ being full, we can apply Proposition \ref{surj} and Proposition \ref{stab connected}
to conclude that $Symp_h(X, \omega)$ is  connnected.


\end{proof}

\section{Proof in the case of $\CC P^2  \# 4{\overline {\CC P^2}}$ }

\subsection{The configuration for $\CC P^2  \# 4{\overline {\CC P^2}}$ }\label{s:conf}
Let  $X= \CC P^2  \# 4{\overline {\CC P^2}}$ and $\omega$ an arbitrary symplectic form on $X$.
We consider a  configuration $C$ in \cite{Eva11},  consisting  of symplectic  spheres in homology classes
$S_{12}=H-E_1-E_2$, $S_{34}=H-E_3-E_4$, $E_1$, $E_2$, $E_3$ and $E_4$. Here $\{H, E_i\}$ is the standard basis of $H_2(X;\ZZ)$ with positive pairing with $\omega$. In Figure $2$ we label the spheres by their
homology classes.

\begin{figure}[h!]
\[
\xy
(0,-20)*{};(30,0)*{} **\dir{-};
(18,0)*{};(48,-20)*{} **\dir{-};
(0,-15)*{};(10,-24)*{} **\dir{-};
(7,-10)*{};(17,-19)*{} **\dir{-};
(48,-15)*{};(38,-24)*{} **\dir{-};
(41,-10)*{};(31,-19)*{} **\dir{-};
(12,-26)*{E_1};(19,-21)*{E_2};
(37,-26)*{E_3};(30,-21)*{E_4};
(14,-6)*{S_{12}};(34,-6)*{S_{34}};
\endxy
\]
\caption{}
\end{figure}

To apply the criterion in Corollary \ref{criterion}, we need to check that we can always find a configuration $C$ of such a homology type, so that \\
$\bullet$ $C$ is stable.\\
$\bullet$ $C$ is  a type I, II or III configuration.\\
$\bullet$ $C$ is full.\\
$\bullet$ $Symp_{c}(U)$ is connected.\\


Existence of such a configuration is a direct consequence of Gromov-Witten theory and the first
three statements follows from definition.  Note also that the actual choice of configuration
will not affect the last statement because $Symp_{h}(X)$ acts transitively on $\mathcal{C}_0$,
which means $U$ is well-defined up to symplectomorphism for any choice of $C\in \mathcal{C}_0$.

It thus remains to prove  the connectedness of $Symp_c(U)$. We will actually show that
$Symp_c(U)$ is weakly contractible in the next subsection.

\subsection{Contractibility of $Symp_{c}(U)$}\label{symcconn}
Let us first recall  the following result of  Evans (Theorem 1.6 in \cite{Eva11}):

\begin{thm}\label{cstarcthm}
If $\CC^*\times\CC$ is equipped with the standard (product) symplectic form $\omega_{std}$ then
$Symp_c(\CC^*\times\CC)$ is weakly contractible.
\end{thm}

This is relevant since Evans  observed in \cite{Eva11}  that,  if $(\omega, J_0)$  is K\"ahler with $\omega$ monotone and $C$  holomorphic,
then  $(U, J_0)$ has a  finite type Stein structure $f$ with $\omega|_U=-dd^cf$,
 and there is a   biholomorphism $\Psi$ from $(U, J_0)$  to $\CC^*\times\CC$
 (In addition,  $\Psi$ satisfies
 $\Psi^*\omega_{std}=\omega|_U$).

 Let us  also recall
the next result of Evans   (Proposition 2.2 in
\cite{Eva11}):

\begin{prp}\label{twostein}
If $(W,J_0)$ is a complex manifold with two finite type Stein structures $ \phi_1 $ and $ \phi_2$, then $Symp_c(W, -dd^c\phi_1) $ and $Symp_c(W, -dd^c\phi_2) $  are weakly
homotopy equivalent.
\end{prp}

Now we complete our proof of the connectedness of $Symp_h(\CC P^2  \# 4{\overline {\CC P^2}}, \omega)$ for an arbitrary $\omega$
by proving the following

\begin{prp}\label{cptconn}
$Symp_c(U, \omega|_U)$ is weakly contractible.
\end{prp}

\begin{proof}
We first choose a specific configuration $C$ convenient for our purpose (as we explained in Section \ref{s:conf} this does not affect our result).
According to \cite{Li08} Proposition 4.8, we can always pick an integrable complex structure $J_0$ compatible with $\omega$, so that $(X, J_0)$ is biholomorphic to a generic blow up of
$4$ points on $\CC P^2$ (the genericity here means that no $3$ points lies on the same line, and indeed this can always be done for less than $9$ point blow ups).
For such a generic holomorphic  blow up, there is a unique smooth rational curve in each class in
the homology type of $C$. Thus we canonically obtain a configuration $C$ associated to $J_0$.
Observe that the complement $U=X\setminus C$ is biholomorphic to
$\CC^* \times \CC$.
That is because the configuration $C$ is the total transformation of two
lines blowing up at four points.
Removing $C$ gives us a  biholomorphism from $(U,J_0)$
to $\CC P^2$ with two lines removed, which is $\CC^* \times \CC$.

Now we construct   a Stein structure $\phi$ on $(U, J_0)$ with $-dd^c\phi=\omega|_U$,
whenever $\omega$  is a  rational symplectic form on $\CC P^2  \# 4{\overline {\CC P^2}}$.
Since $(U, J_0)$ is biholomorphic to $\CC^* \times \CC$ equipped with the standard finite type Stein structure $(J_{std},\omega_{std}=-dd^c|z|^2)$,
we  can then apply Proposition \ref{twostein} and Theorem \ref{cstarcthm} in this case to conclude the weak contractibility of  $Symp_c(U,\omega|_U)$.


So we assume that $[\omega] \in H^2(X;\QQ)$. Up to rescaling,  we can write  $PD([l\omega])=aH-b_1E_1-b_2E_2-b_3E_3-b_4E_4$ with  $ a, b_i \in \ZZ^{\geq0}$. Further, we assume $ b_1\geq b_2, b_3 \geq b_4$.
 Since $H-E_1- E_3$ is an exceptional class we also have $\omega(H-E_1- E_3)>0$.
 This means that  $ a>b_1+b_3$, namely, $2a \geq 2b_1+2b_3+2$.
Rewrite $$PD([2l\omega])=(2b_1 +1)(H-E_1-E_2)+
E_1+(2b_1-2b_2+1)E_2+(2a-2b_1-1)(H-E_3-E_4)$$
$$+(2a-1-2b_1-2b_3)E_3+(2a-1-2b_1-b_4)E_4.$$
Notice that  the coefficients are all in $\ZZ^{\textgreater 0}$. In this way we represent $PD([2l\omega])$  as a
positive integral combination of all elements in the set $\{H-E_1-E_2, H-E_3-E_4, E_1,E_2,E_3, E_4\}$, which is the homology type of $C$.

Denote the symplectic  sphere with homology class $E_i$  in $C$ by $C_{E_i}$, and similarly for
the two remaining spheres.
Notice that each sphere  is a smooth divisor.  Consider  the effective divisor
$$F=(2b_1 +1)C_{H-E_1-E_2}+ C_{ E_1 }+(2b_1-2b_2+1)C_{ E_2}+(2a-2b_1-1)C_{H-E_3-E_4}$$
$$+(2a-1-2b_1-2b_3)C_{E_3}+(2a-1-2b_1-b_4)C_{E_4}.$$ There
is a holomorphic line bundle $\mathcal L$ with a  holomorphic  section
$s$ whose zero divisor is exactly $F$. Notice that $$c_1(\mathcal L)=[F]=[2l\omega].$$
By \cite{GH94} section 1.2, we can take an hermitian metric $|\cdot|$ and a compatible connection on $\mathcal L$
such that the curvature form is just
$2l\omega$.
Moreover, for the holomorphic section $s$,  the fuction $\phi=-log|s|^2$ is plurisubharmonic on the complement $U$
with $-d(d\phi\circ J_0)=2l\omega$. Notice that $F$ and $C$ have the same support so the complement of $F$ is the same as $U$.
Thus we have endowed  $(U,J_0)$ with a  finite type Stein structure $\phi$.

As argued above,  this implies that  $Symp_c(U,\omega|_U)=Symp_c(U,2l\omega|_U)$ is  weakly contractible
 when $[\omega] \in H^2(X,\QQ)$ by the biholomorphism from $(U,J_0)$ to $(\mathcal{C}^*\times \mathcal{C},J_{std})$.



Finally, suppose $\omega $ is not rational, but we assume
$\omega(H)\in \mathbb Q$ without loss of generality by rescaling. We
take a base point $\varphi_0 \in Symp_c(U, \omega|_U)$, and a
$S^n(n\geq0)$ family  of symplectomorphisms determined by a based
map $\iota: S^n \rightarrow Symp_c(U, \omega '|_U)$. Denote the
union of support of this $S^n$ family by $V_{\iota}$, which is a
compact subset of $U$.

Note the following fact:

\begin{clm}  There exists an $ \omega '$ symplectic on $X$ such that:
\begin{enumerate}
\item $[\omega '] \in H^2(X,\QQ) $,
\item $[\omega '] (E_i)\geq [\omega] (E_i), [\omega '] (H)=[\omega] (H)$
\item the configuration $C$ is $\omega '-$ symplectic
\item $(X\setminus C, \omega ')\hookrightarrow (X\setminus C, \omega)$ in such a way that
the image contains $V_{\iota}$.
\end{enumerate}
\end{clm}

\begin{proof}
Recall that to blow up an embeded ball $B$ in a symplectic manifold
$(M,\omega)$, one removes the ball and collapses the boundary by
Hopf fibration which incurs an exceptional divisor.  The reverse of
this procedure is a blowdown.

Now take $E_i$ in the configuration $C$ and blow them down to get a
disjoint union of balls $B_i$ in the blown-down manifold, which is a
symplectic $\CC P^2$ with line area equal $\w(H)$. One then enlarge
$B_i$ by a very small amount to $B'_i$ so that the sizes of $B'_i$
become rational numbers. After the enlargement, blow up $B'_i$.
This produces a symplectic form on $X$ which clearly satisfies (1)
and (2). (3) can be achieved as long as the enlarged ball has
boundary intersecting proper transformation of $S_{12}$ and $S_{34}$
on a big circle.  This is always possible: perturb $S_{12}$ and
$S_{34}$ slightly so that they are symplectically orthogonal to $E_i$
before blow-down.  Then in a neigbhorhoold of the resulting balls $B_i$,
one has a Darboux chart where $B_i$ is the standard ball, while
the portion of $S_{12}$ and $S_{34}$ inside this chart is the $x_1-x_2$ plane.
This is guaranteed by symplectic neighborhood
theorem near $E_i$.  Hence the (3) is obtained when the enlargement stays 
inside the Darboux chart.  For more details one is referred to \cite{MW96}.

To see (4), we note that from the above description, $(X\setminus C,
\omega')$ is symplectomorphic to the complement of $\bigcup_i
B'_i$ union two lines (the proper transforms of
$S_{12}$ and $S_{34}$) in the symplectic $\CC P^2$ from blowing down.
The same thus applies to $(X\setminus C,
\omega)$, while $B_i'$ are replaced by $B_i\subset B_i'$. Therefore,
the statement regarding embedding holds in (4).  Since $V_\iota$ is
compact and embeds in $(X\setminus C,\w)$, as long as the amount of enlargement from $B_i$ to $B'_i$
is small enough, the embedded
image contains $V_{\iota}$ as claimed.
\end{proof}

Therefore
we can find an isotopy in $Symp_c(U,\omega '|_U) \hookrightarrow Symp_c(U,\omega|_U)$,  from the $S^n$ family of maps
to the base point $\varphi_0$ by the proved case when $\w$ is rational. We emphasize in the above proof, the choice of $\w'$
depends on $\iota$, but this is irrelevant for our purpose.
This concludes that for arbitrary symplectic form $\omega$ on $X$, $Symp_c(U,\omega|_U)$
is weakly contractible and hence $Symp_h(\CC P^2  \# 4{\overline {\CC P^2}})$ is connected
for any symplectic form.

\end{proof}

\begin{rmk} \label{kleq3}
The approach we adopt in this note in fact  provides a uniform way to establish  the connectedness of the Torelli part of
SMC for all  symplectic rational $4-$manifolds with $\chi \leq 7$.
This can be viewed as a continuation  of the techniques first introduced by Gromov in \cite{Gro85} 
and further  developed by  many others in \cite{Abr98,AM99,LP04,Eva11,AP12} etc.

Here we just list the configurations  for the 1,2,3-point blow up of $\CC P^2$ equipped
 with an  arbitrary symplectic form:

 \begin{itemize}

 \item $ \CC P^2  \# {\overline {\CC P^2}}$,    $\{E_1,  H-E_1 \hbox{(with a marked point)}\}$.

\item  $\CC P^2  \# 2{\overline {\CC P^2}}$, $\{E_1,E_2,H-E_1-E_2\}$.

\item $\CC P^2  \# 3{\overline {\CC P^2}}$,  $\{E_1, E_2, H-E_1-E_2, H-E_1-E_3, H-E_2-E_3 \}$.
 \end{itemize}
The  configurations  are all of type I.
Combined with our argument verbatim,
 one can  recover the connectedness of $ Symp_h(\CC P^2  \# n{\overline {\CC P^2}}, \omega)$,  $n\leq3$.
 However, such a result for these manifolds is not new, see \cite{Abr98,AM99,LP04,Eva11}.

\end{rmk}







\begin{thebibliography}{Gom95}

\bibitem[Abr98]{Abr98}
Miguel Abreu.
\newblock Topology of symplectomorphism groups of {$S^2\times S^2$}.
\newblock {\em Inventiones Mathematicae}, 131:1--23, 1998.

\bibitem[AM99]{AM99}
Miguel Abreu, Dusa Mcduff.
\newblock Topology of symplectomorphism groups of rational ruled surfaces.
\newblock {\em 	J. Amer. Math. Soc}, 971--1009, 1999.

\bibitem[Anj02]{Anj02}
S{\'{\i}}lvia Anjos.
\newblock Homotopy type of symplectomorphism groups of {$S^2\times S^2$}.
\newblock {\em Geom. Topol.}, 195--218 (electronic), 2002.

\bibitem[AP12]{AP12}
S{\'{\i}}lvia  Anjos, Martin Pinsonnault.
\newblock The homotopy Lie algebra of symplectomorphism groups of 3-fold blow-ups of the projective plane.
\newblock {\em Math. Z}, 275, no. 1-2, 245-292, 2013.

\bibitem[BCi01]{BCi01}
Paul Biran, Kai Cieliebak.
\newblock Symplectic topology on subcritical manifolds.
\newblock {\em Comm. Math. Helv}, 76:712--753, 2001.

\bibitem[BLW12]{BLW12}
Matthew Strom Borman, Tian-Jun Li, Weiwei Wu.
\newblock Spherical Lagrangians via ball packings and symplectic cutting.
\newblock {\em Selecta Mathematica}, 20, no. 1, 261-283, 2014.

\bibitem[Eva11]{Eva11}
Jonathan Evans.
\newblock Symplectic mapping class groups of some Stein and rational surfaces.
\newblock {\em Journal of Symplectic Geometry},   9(1):45-82, 2011.

\bibitem[GH94]{GH94}
Phillip Griffiths, Joseph Harris.
\newblock {\em Principles of {Algebraic Geometry}}.
\newblock Wiley Interscience, {Wiley Classics Library}, edition, 1994.

\bibitem[Gom95]{Gom95}
Robert Gompf.
\newblock A new construction of symplectic manifolds.
\newblock {\em Annals of Mathematics}, 142(3):527--595, 1995.

\bibitem[Gro85]{Gro85}
Misha Gromov.
\newblock Pseudoholomorphic curves in symplectic manifolds.
\newblock {\em Inventiones Mathematicae}, 82:307--347, 1985.

\bibitem[Hin03]{Hin03}
Richard Hind.
\newblock Stein fillings of lens spaces.
\newblock {\em Communications in Contemporary Mathematics}, 5:967--982, 2003.

\bibitem[Hir76]{Hir76}
Morris Hirsch.
\newblock {\em Differential Topology}, volume~33 of {\em Graduate Texts in
  Mathematics}.
\newblock Springer, 1976.

\bibitem[Li08]{Li08}
Tian-jun Li.
\newblock {\em The space of Symplectic structure on closed
  4-manifolds}.
\newblock {\em Studies in AMS}, Vol42, 259--277,2008

\bibitem[LP04]{LP04}
Francois Lalonde, Martin Pinsonnault.
\newblock {\em The topology of the space of symplectic balls in rational
  4-manifolds}.
\newblock {\em Duke Mathematical Journal}, 122(2):347--397, 2004.

\bibitem[LW11]{LW11}
Tian-Jun Li, Weiwei Wu.
\newblock {\em Lagrangian spheres, symplectic surface and the symplectic mapping
class group}.
\newblock Geometry and Topology, 16(2):1121-1169, 2012.

\bibitem[MS04]{MS04}
Dusa McDuff, Dietmar Salamon.
\newblock {\em J-holomorphic curves and symplectic topology}.
\newblock  volume~52 of {\em Colloquium Publications}, American Mathematical Society, 2004.

\bibitem[MS05]{MS05}
Dusa McDuff, Dietmar Salamon.
\newblock {\em Introduction to Symplectic Topology}.
\newblock Oxford Mathematical Monographs, second edition, 2005.

\bibitem[MW96]{MW96}
McCarthy, John D.; Wolfson, Jon G.
\newblock {\em Double points and the proper transform in symplectic geometry.}.
\newblock Differential Geom. Appl. 6 (1996), no. 2, 10107.

\bibitem[Pai60]{Pai60}
Richard S. Palais.
\newblock {\em On the local triviality of the restriction map for embeddings}.
\newblock Comm. Math. Helv. 34 (1960), 306-312.

\bibitem[Ruan93] {Ruan93}
Yongbin Ruan.
\newblock{\em Symplectic topology and extremal rays}.
\newblock GAFA 3 (4), pp 395-430,1993.


\bibitem[HW13]{HW13}
Richard Hind, Weiwei Wu.
\newblock {\em Symplectormophism groups of non-compact manifolds and space of Lagrangians}
\newblock {J. Symplectic Geom., to appear, arXiv:1305.7291}.

\bibitem[Sei08]{Sei08}
Paul Seidel.
\newblock Lectures on four-dimensional {D}ehn twists.
\newblock In {\em Symplectic 4-Manifolds and Algebraic Surfaces}.
\newblock volume 1938
  of {\em Lecture {N}otes in {M}athematics},  231--268. Springer, 2008.


\bibitem[SS06]{SS06}
Paul Seidel, Ivan Smith.
\newblock The symplectic topology of Ramanujam's surface.
\newblock In {\em Symplectic 4-Manifolds and Algebraic Surfaces}.
\newblock volume 1938
  of {\em  Comm. Math. Helv}, 80, no.4:859--881, 2006.





\end{thebibliography}


\end{document}